\documentclass[a4paper,10pt]{amsart}

\usepackage[utf8]{inputenc}
\usepackage[usenames, dvipsnames]{color}
\newtheorem{theorem}{Theorem}
\newtheorem{definition}[theorem]{Definition}
\newtheorem{lemma}[theorem]{Lemma}
\newtheorem{corollary}[theorem]{Corollary}

\newtheorem{remark}[theorem]{Remark}
\newtheorem{conjecture}[theorem]{Conjecture}
\newcommand{\C}{\mathbb C}
\newcommand{\R}{\mathbb R}
\newcommand{\N}{\mathbb N}

\title[Approximation in measure]{Approximation in measure: Dirichlet problem, universality and the Riemann hypothesis}

\dedicatory{To Anatoly Georgevich from the junior author, to Tolya from the senior author}

\author[J. Falc\'o]{Javier Falc\'o}
\address[Javier Falc\'o]{Departamento de An\'alisis Matem\'atico, Universidad de Valencia, Doctor Moliner 50, 46100 Burjasot (Valencia),
Spain.}
\email{Francisco.J.Falco@uv.es}

\author[P. M. Gauthier]{Paul M. Gauthier}
\address[Paul M. Gauthier]{D\'epartement de math\'ematiques et de statistique, Universit\'e de Montr\'eal, Montr\'eal, Qu\'ebec, Canada H3C3J7.}
\email{gauthier@dms.umontreal.ca}

\address{}
\email{}

\begin{document}

\begin{abstract}
Approximation in measure is employed to solve an asymptotic Dirichlet problem on arbitrary open sets and to show that many functions, including the Riemann zeta-function, are universal in measure. Connections with the Riemann Hypothesis are suggested.
\end{abstract}


\thanks{First author was supported by MINECO and FEDER Project MTM2017-83262-C2-1-P. Second author was supported by NSERC (Canada) grant RGPIN-2016-04107.}

\maketitle

\section{Introduction}

In his beautiful survey on universality, Karl-Goswin Grosse-Erdman \cite[p. 263]{G-E99} remarks that
the notion of universality in measure was already implicit in the work of Mens'shov in 1945 on universal trigonometric series. Recently, approximation in measure by holomorphic and harmonic functions was considered in \cite{GS17} and \cite{GS19}. In the present paper, we show the existence and preponderance of holomorphic and harmonic functions which are universal with respect to approximation in measure. 

We employ approximation in measure to solve a Dirichlet type problem for harmonic functions on arbitrary open subsets of Riemannian manifolds. The Dirichlet type problem we solve is a Dirichlet problem {\em in measure}, not with respect to the measure on the boundary, which is quite classical, but rather with respect to the  measure on the open set in the sense of Definition \ref{boundarydensity}. Similar results where obtained in \cite{FalcoGauthier} for holomorphic functions on arbitrary open sets of Riemann surfaces.

We also study universality in measure for harmonic functions on $\mathbb R^n$ and universality in measure for Dirichlet series on the complex plane.

On every Riemannian manifold $M$ there is a natural volume measure and we denote the volume of a Borel subset $A$ of $M$ by $\mu(A).$ We denote a ball of center $p$ and radius $r$ by $B(p,r).$ For a locally compact Hausdorff space $X,$ we denote the ideal point of the one-point compactification by $*_X$ or, if there is no ambiguity, simply by $*.$ 
A subset of an open set $U$ is said to be bounded in $U$ if its closure in $U$ is compact. 
By a regular exhaustion of $X,$ we mean a sequence $\{X_n\}$ of compact subsets of $X,$ such that, for each $n,$ $X_n\subset X_{n+1}^0$ and $X^*\setminus X_n$ is connected.

For a subset $F\subset M,$ we denote by $H(F)$ the family of all functions $u$ which are harmonic on some (depending on $u$) neighbourhood of $F.$

\begin{definition}
\label{boundarydensity}
For a boundary point $p$ of an open  subset $U\subset M$  and a measurable set $A\subset U,$ we define the {\em density} of $A$ relative to $U$ at $p$ as 
$$
	\mu_U(A,p) = \lim_{r\searrow 0}\frac{\mu\big(A\cap B(p,r)\big)}{\mu\big(U\cap B(p,r)\big)}
$$
 if this limit exists.
\end{definition} 

Obviously, 
$$
	\mu_U(A,p)=0 \quad \mbox{if and only if} \quad  \mu_U(U\setminus A,p)=1.
$$

%
%
%
%

As we mentioned before, our main goal in this note is to solve a Dirichlet type problem in measure in the sense of the previous definition.

The structure of this manuscript is as follows. In Section \ref{AsymptoticDirichletproblem} we will prove our asymptotic Dirichlet result in the sense of Definition \ref{boundarydensity}. To be more specific, we will show that fixed an open subset of a Riemannian manifold and   a continuous function on the boundary of the set, there exists a  harmonic (or holomorphic) function defined on the set that approaches the original function except on a set of density 0 at every boundary point. In Section \ref{universalityinmeasure} we study harmonic universality in measure in the sense of Birkhoff. Our last section is devoted to the study of universality in measure for functions  that  admit a representation as a Dirichlet series in some half-plane. In particular we obtain  a universality theorem in measure for the Riemann’s zeta-function.
Andersson \cite{A13} has recently shown the remarkable fact that a certain refinement of Voronins's spectacular Theorem on the universality of the Riemann zeta-function, would be equivalent to a natural refinement of the most important theorem in complex approximation, Mergelyan's theorem. The Selberg class $S$ is an attempt to axiomatize a natural class of Dirichlet series  for which an analogue of the Riemann-Hypothesis would hold. Steuding \cite{S} has introduced a class $\tilde S$ similar to the Selberg class and to which the Universality Theorem extends. We extend Andersson's equivalence theorem to functions in the intersection $S\cap\tilde S$  of the Selberg class and the Steuding class $\tilde S,$ which includes many Dirichlet series which one naturally encounters, including the Riemann zeta-function.



\section{Asymptotic Dirichlet problem}
\label{AsymptoticDirichletproblem}

Our main result in this section is Theorem \ref{DirichletTheorem}. That is the asymptotic Dirichlet result already mentioned in the introduction. To obtain this result we will use Runge-Carleman sets. We start by recalling the definition of these sets.

\begin{definition}
Let $F\subset U\subset M,$ where $M$ is a Riemannian manifold, $U$ is open and $F$ is closed relative to $U.$
The set $F$ is called a  Runge-Carleman set in $U,$ if  for every $u\in H(F)$  and every positive  continuous function $\epsilon$ on $F,$  there exists a harmonic function $h:U\mapsto \mathbb R$ such that $|h-u |<\epsilon$ on $F.$
\end{definition}

Our first result is an observation that extends \cite[Lemma 13]{GS19}. This lemma gives a sufficient condition for a set to be a Runge-Carleman set in terms of local connectivity and local finiteness. A set is called locally connected if every point admits a neighbourhood basis consisting entirely of open, connected sets. A collection of sets $E_j$ on a Riemannian manifold $M$  is called locally finite if  every compact set $K$ in $M$ intersects only finitely many of the sets $E_j$.

\begin{theorem}\label{Lemma 13}
Let $F\subset U\subset M,$ where $M$ is a Riemannian manifold, $U$ is open and $F$ is closed relative to $U$. Assume $U^*\setminus F$ is connected and locally connected and $F$ is the union of a  family of disjoint compact sets, which in $U$ is locally finite. Then, $F$ is a  Runge-Carleman set in $U.$
\end{theorem}

\begin{proof}
If $U=M,$ this is  \cite[Lemma 13]{GS19}. The general case follows by applying Lemma 13 to each connected component of $U,$ since each connected component is a Riemannian manifold.  
\end{proof}



Our next result is a geometric lemma that combined with Theorem \ref{Lemma 13} will be crucial in the proof of Theorem \ref{DirichletTheorem}

\begin{lemma}
\label{smallremoval}
Let $U$ be a proper open subset of a Riemannian  manifold $M$ and $s$ a  connected  compact subset of $U$ of zero volume. Then, for each   $\epsilon>0$  and $r>0,$ there is a connected open neighbourhood $R$ of $s$ in $U$ such that 
\[
	\frac{\mu\big(R\cap  B(p,r)\big)}{\mu\big(U\cap B(p,r)\big)} < \epsilon, 
\]
for every $p\in \partial U$.\\

\end{lemma}

\begin{proof}

Consider a compact set $s'$ in $U$ such that $s$ is contained in the interior of $s'$. Since $s'$ is a compact subset of $U$ we can find a positive number $h$ with $d(s',\partial U)>2h>0$  and hence $\overline{B(x,4h/3)}$ compact for all $x\in s'$. Let 
\[
	V(h):=\inf\{\mu\big(B(\alpha,h)\cap B(\beta,h)\big):\alpha\in s', \beta\in U, d(\alpha,\beta)=h\}.
\]
We claim that $V(h)>0$. 
Indeed, if $V(h)=0$ we could find sequences of points $\{\alpha_{n}\}_{n}$ and $\{\beta_{n}\}_{n}$ such that $\mu\big(B(\alpha_{n},h)\cap B(\beta_{n},h)\big)\to 0$ and $d(\alpha_{n}, \beta_{n})=h$. By the compactness of $s'$ we can assume that $\alpha_{n}\to \alpha\in s'$. Then, for every $n$ bigger than some natural number $n_{0}$ we have that $d(\alpha_{n},\alpha)<h/4$. Since $d(\alpha,\beta_n)\le d(\alpha,\alpha_n)+d(\alpha_n,\beta_n),$  
for all $n>n_{0}$ we have that  $\beta_{n}\in {B(\alpha, 5h/4)}\subset U$. By compactness  of $\overline{B(\alpha, 4h/3)}$ we can assume that $\beta_{n}\to \beta\in  \overline{B(\alpha, 5h/4)}\subset  U$.  We may also assume that $d(\beta_n,\beta)<h/4.$ 
Now, there is a path $\sigma$ of length $\ell(\sigma)<4h/3$ in $M$ going from $\alpha$ to $\beta.$ Since $d(s^\prime,\partial U)>2h,$ the path $\sigma$ lies in $U.$   
As a point $\gamma$ moves along this path, we write $\sigma=\sigma_{<\gamma}+\sigma_{>\gamma},$ where $\sigma_{<\gamma}$ is the portion of $\sigma$ from $\alpha$ to $\gamma$ and $\sigma_{>\gamma}$ is the portion of $\sigma$ from $\gamma$ to $\beta.$ 
$\ell(\sigma_{<\gamma})+\ell(\sigma_{>\gamma})=\ell(\sigma)<4h/3,$ The function $\ell(\sigma_{<\gamma})- \ell(\sigma_{>\gamma})$ depending on $\gamma$ has a zero for some $\gamma$  because it is continuous and for $\gamma=\alpha$ is negative while for $\gamma=\beta$ it  is positive.
This point $\gamma$ cuts the path into two paths of equal length, so $d(\alpha,\gamma)<2h/3$ and $d(\gamma,\beta)<2h/3.$ 
By the triangle inequality we have that for all $n>n_{0}$
\[
	B(\gamma,h/12)\subset B(\alpha_{n},h)\cap B(\beta_{n},h).
\]
But $\mu\big(B(\alpha_{n},h)\cap B(\beta_{n},h)\big)\to 0$ when $n$ goes to infinity and $\mu\big(B(\gamma,h/12)\big)>0$ (since the Riemannian volume of every open set is positive),  which is a contradiction. Therefore $V(h)>0$.

 Since the measure  $\mu$ is regular, we may choose $R$  to be a connected open neighbourhood of $s$ in $U$ that is contained in an open $\epsilon$-neighbourhood of $s$  contained in $s'$ and with measure smaller than $\epsilon V(h),$ where $\epsilon< d(s',\partial U)-2h$. 
Note that $d(R,\partial U)>2h$. We claim that such a set satisfies the result.

Fix a point $p\in \partial U$ and $r>0$. If $R\cap  B(p,r)=\emptyset,$ then the result holds automatically, so we can assume that $R\cap  B(p,r)\ne \emptyset$.  Thus, $r>2h.$ Fix $x\in R\cap  B(p,r)$. We can find a path $\sigma$ in $M$ joining $x$ and $p$  with $2h<\ell(\sigma)<r.$ Then, we can find a 
point $y$ in the path such that  $d(x,y)=h$. Since $U$ is open and $d(x,\partial U)\geq d(R,\partial U)>2h$ we have that $y$ is in $U$. Also, we have that $d(y,p)< \ell(\sigma)-d(x,y) < r-h$. Thus, $B(y,h)\subset B(p,r)$.
 
 Note that, since $d(R,\partial U)>2h$, we have that  $B(x,h)\subset U$. Using also that  $B(y,h)\subset B(p,r)$, $x\in s'$, $d(x,y)=h$ and the definition of $V(h)$, we obtain
\[
	\mu\big(U\cap B(p,r)\big)
\geq \mu\big(B(x,h)\cap B(p,r)\big) 
\geq \mu\big(B(x,h)\cap B(y,h)\big) 
\geq V(h). 
\]

Hence,
\[
	\frac{\mu\big(R\cap  B(p,r)\big)}{\mu\big(U\cap B(p,r)\big)} 
	\leq \frac{\mu\big(R\cap  B(p,r)\big)}{V(h)} 
	\leq \frac{\mu\big(R\big)}{V(h)} 
	< \frac{\epsilon V(h)}{V(h)}
	=\epsilon.
\]

\end{proof}

Now we present our main result in this section.

\begin{theorem}
\label{DirichletTheorem}
Let $U$ be an open subset of a Riemannian manifold $M$ and $\varphi$  a continuous function on $\partial U.$  Then, there exists a  harmonic function $\widetilde\varphi$ on $U,$  such that, for every $p\in \partial U$, $\widetilde\varphi(x)\to \varphi(p)$, as $x\to p$ outside a set of density 0 at $p$ relative to $U.$  
\end{theorem}

\begin{proof}

If $U=M,$ then $\partial U=\emptyset$ so there is nothing to prove. On the other hand, if $U\not=M,$ then $\partial U\not=\emptyset$ (since $M$ is connected).

First, we extend $\varphi$ continuously to $\overline U$ and retain the same notation for the extension.

	Let $\mathcal S=\{S_j\}_{j=0}^{\infty}$ be a locally finite family of smoothly bounded compact parametric balls $S$ in $U$ such that $U=\cup_jS_j^0$ and $|S|<dist(S,\partial U),$ where $|S|$ denotes the diameter of $S$.  Without loss of generality we may assume that none of these balls contains another. Let $s_j=\partial S_j.$   By  Lemma \ref{smallremoval} there is an open neighbourhood $R_j$ of $s_j$ in $U$ such that 
$$
	\frac{\mu\big(R_j\cap  B(p,r)\big)}{\mu\big(U\cap B(p,r)\big)} < \frac{1}{2^j}, \quad \forall \, p\in  \partial U. 
$$ 
We may assume that each $R_j$ is a smoothly bounded shell. That is, that in a local coordinate system, $R_j=\{x: \rho_j<\|x\|<1\}.$ We may also assume that  $|R_j|<dist(R_j,\partial U)$ and  the family $\{R_j\}$ is locally finite (but note that the family is unbounded). Set $R=R_0\cup R_1\cup\cdots.$ Since no ball contains another ball, if two balls meet then their boundaries meet and consequently their shells meet.

We claim that $R\cup \{*_U\}$ is connected. Since $R$ is unbounded in $U$ it is sufficient to show that $R$ is connected. By a standard argument, for every two balls $S_i$ and $S_j$ in  $\mathcal S,$ there is a finite chain $S_{n_k}, \, k=1,\ldots, \ell$ in $\mathcal S,$ such that $S_{n_1}=S_i,$ $S_{n_\ell}=S_j,$ and
$\partial S_{n_k}\cap \partial S_{n_{k+1}}\not=\emptyset$, 
for $k=1,\ldots, \ell-1.$ This yields a corresponding chain of shells  $R_{n_k}$ from $R_i$ to $R_j.$ Thus, $R$ is indeed connected, so $R\cup \{*_U\}$ is connected. 

We claim that $R\cup \{*_U\}$ is also {\em locally} connected. Since $R$ is clearly locally connected, it is sufficient to show that $R\cup \{*_U\}$ is locally connected at the ideal point $*_U.$ For this it is sufficient to construct a neighbourhood base $\{V_n\}_{n=1}^\infty$ of $*_U$ in $R\cup \{*_U\}$  consisting of connected sets. Let $\{K_n\}_{n=0}^\infty$ be a regular exhaustion of $U$ by smoothly bounded compact sets. 
Thus, each $U^*\setminus K_n$ is connected. 
For each $n,$ let 
$$
	Q_n=K_n\setminus \bigcup\{S^0_j: S^0_j\cap \partial K_n\not=\emptyset\}.
$$
Since the $S_j's$ form a locally finite family in $U,$ 
$\{Q_n\}$ is also an exhaustion of $U$ by compact sets. 
$$
	U^*\setminus Q_n = (U^*\setminus K_n)\cup \bigcup\{S^0_j: S^0_j\cap \partial K_n\not=\emptyset\}.
$$
Since $U^*\setminus K_n$ is connected and each $S_j^0$ in this union meets $U^*\setminus K_n,$ we have that $U^*\setminus Q_n$ is also connected. Thus, the exhaustion $\{Q_n\}_{n=1}^\infty$ of $U$ is also regular. Since the family $\{Q_n\}_{n=1}^\infty$ is locally finite in $U,$  
$$ 
	V_n=(R\cup \{*_U\})\setminus Q_n=(R\setminus Q_n)\cup\{*_U\}
$$
is a neighbourhood base of $*_U$ in $R\cup \{*_U\}.$ We wish to show that each $V_n$ is connected. 
The boundary $\partial K_n$ of $K_n$ has finitely many connected components, each of which is the boundary of a complementary connected component $W_i, \, i=1,\ldots,m$ of $K_n$ and each $W_i$
is unbounded  in $U$  (since $U^*\setminus K_n$ is connected). Set
$$
	V_{n,i}= \{*_U\}\cup\bigcup\{R_j: S^0_j\cap\overline W_i\not=\emptyset\}.
$$
Since the boundary of the connected components of $K_n$ are the boundaries of the  $W_i`s,$ we have $V_n=V_{n,1}\cup\cdots\cup V_{n,m}$ To show the connectivity of $V_n,$ it is sufficient to show the connectivity of each $V_{n,i},$ since they all have the point $*_U$ in common. 
Hence, it is sufficient to show that 
\begin{equation}\label{component}
	\bigcup\{R_j: S^0_j\cap\overline W_i\not=\emptyset\}
\end{equation}
is connected, since the latter is unbounded. Thus, it is sufficient to show that for every two shells $R_k$ and $R_\ell$  such that $S^0_k\cap \overline W_i\not=\emptyset$ and $S^0_\ell\cap \overline W_i\not=\emptyset$  there is a finite chain of such shells connecting them. We do know that there is a finite chain of shells connecting $R_k$ and $R_\ell$,  possibly without the condition that the corresponding balls meet $\overline W_i$. Suppose such a connecting chain does not have this intersection property. We can replace it by a chain that does have the required intersection property as follows. The first shell in the chain is $R_k$ and it does have the required intersection property. Let $R_{\nu}$ be the last consecutive shell following $R_k$ having the required intersection property.  Since $R_{\ell}$ also has the required intersection property, there is a last consecutive $R_{\mu}$ following $R_\nu$ not having the required intersection property. Since $R_\nu$ and $R_{\mu+1}$ have the property that their corresponding balls $S^0_\nu$ and $S^0_{\mu+1}$ intersect $\partial W_i$ which is path-connected, we can replace the subchain from $R_\nu$ to $R_{\mu+1}$ by a finite chain of shells, whose corresponding balls intersect $\partial W_i.$ Since the initial chain from $R_k$ to $R_\ell$ is finite, after finitely many of such replacements, we end up with a finite chain of shells all having the required intersection property, showing that  (\ref{component}) is indeed connected and therefore $V_{n,i}$ is connected. Since this holds for each $V_{n,i}$ and all of them have the point $*_U$ in common, their union $V_n$ is also connected. This completes the proof of  the claim that $R\cup\{*_U\}$ is locally connected.

Let $F$ be the closed set $U\setminus R.$ We wish to show that $F$ satisfies the hypotheses of Theorem \ref{Lemma 13}. As a first step, we have shown that $U^*\setminus F=R\cup\{*_U\}$ is connected and locally connected. Recalling the construction of the shells $R_j,$ we may write each shell $R_j$ in the form $R_j= B_j\setminus S_j,$ where $B_j$ is an open topological ball containing the closed ball $S_j.$ Setting $F_0=F\cap S_0$ and $F_j=(F\cap S_j)\setminus (B_0\cup\cdots\cup B_{j-1}),$ for $j>0,$ we may write $F=F_0\cup F_1\cup \cdots.$ The family $\{F_j\}_{j=0}^\infty$ is locally finite and consists of compact sets. Thus, $F$ indeed satisfies the hypotheses of Theorem \ref{Lemma 13}.

Now, choose a continuous function $\varepsilon:F\to (0,1]$ such that $\varepsilon(x)\to 0,$  when $x\to\partial U.$   
From the construction of $F,$ since $F_j\subset S_j$ for $j=0,1,2,\ldots$, we see that$\{F_j\}_{j=1}^\infty$ is  a locally finite family  of disjoint compact sets with 
$|F_j|<d(F_j,\partial U).$

For each $F_{j}$ of $F$ we choose a point $x_{F_{j}}\in F_{j}$ and define a function $g$ on $F$ as $g(x)=\sum_{j\in\mathbb N}\varphi(x_{F_{j}})\chi_{F_{j}}(x).$  Clearly the function $g$ is harmonic on $F$ since $g$ is constant on each $F_j.$ Since $F$ is a Runge-Carleman set in $U,$ we can approximate $g$ by a harmonic function $\widetilde\varphi$ on $U$ with $\vert g(x)-\widetilde\varphi(x)\vert<\varepsilon(x)$ for all $x\in F$.
			
	We claim that $\widetilde\varphi(x)$ satisfies the theorem. Fix $p\in \partial U$ and consider a sequence $\{x_{k}\}_{k=1}^{\infty}$ in $F$ that converges to $p$. Then, $x_{k}\in F_{j_{k}}$ for some $F_{j_{k}}\subset F$. 
	Consider the sequence $\{x_{F_{j_{k}}}\}_{k=1}^{\infty}$ where $x_{F_{j_{k}}}$ is the fixed element in $F_{j_{k}}$ defining $g$. We have that $\vert F_{j_{k}}\vert\to 0$, so
	\[
		p=\lim_{k\to \infty}x_{k}=\lim_{k\to\infty} x_{F_{j_{k}}}
	\]
	and by the continuity of $\varphi$ at $p$ we have that  
	\begin{equation*}
		\limsup_{k\to\infty}\vert \widetilde \varphi(x_{k})-\varphi(p)\vert
		\leq\limsup_{k\to\infty}\big(\vert g(x_{k})-\varphi(p)\vert+\epsilon(x_{k})\big)
		=\limsup_{k\to\infty}\vert \varphi(x_{F_{j_{k}}})-\varphi(p)\vert
		=0.
	\end{equation*}

	It only remains to show that $F$ satisfies that, for every $p\in \partial U$, 
	$$
	\mu_{U}(F,p)=\liminf_{r\rightarrow 0}\frac{\mu\big(B(p,r)\cap F\big)}{\mu\big(B(p,r)\cap U\big)}=1.
	$$ 
For this we shall show that
$$
	\limsup_{r\rightarrow 0}\frac{\mu(B(p,r)\cap(U\setminus F))}{\mu(B(p,r)\cap U)}=0.
$$
Fix $p\in \partial U$ and $\epsilon>0.$ Choose $j_\epsilon$ so that
$$
	\sum_{j\ge j_\epsilon} 2^{-j}<\epsilon.
$$
Choose $r_\epsilon>0$ such that $B(p,r_\epsilon)$ is disjoint from the  neigbourhoods $R_j$ of the segments $s_j$ for $j\le j_\epsilon.$ Then, for all $r<r_\epsilon$, 
 since $(U\setminus F) {=} \cup_j R_j$,
\begin{equation*}
	\frac{\mu(B(p,r)\cap(U\setminus F))}{\mu(B(p,r)\cap U)}
	 = \frac{\mu(B(p,r)\cap(\cup_{j} R_j))}{\mu(B(p,r)\cap U)}\le\sum_{j>j_\epsilon}\frac{\mu(B(p,r)\cap R_j)}{\mu(B(p,r)\cap U)}	\le\sum_{j>j_\epsilon}2^{-j} <\epsilon
	\end{equation*}
	Thus, the $\mu$-density of $U\setminus F$  relative to $U$ at $p$ is at most $\epsilon.$ Since $p$ and $\epsilon$ are arbitrary, this proves the result.

\end{proof}


\section{Harmonic universality in measure}
\label{universalityinmeasure}

We start this section by recalling the following theorem about approximation in measure on  non-compact Riemannian manifolds.

\begin{theorem}[{\cite[Theorem 15]{GS19}}]\label{GS}
 For every measurable subset $E$ of a non-compact Riemannian manifold $M,$ and for every
measurable function $v : E\rightarrow \mathbb R$ there exists a sequence $u_j\in H(M)$ such that $u_j \rightarrow v$ in measure.
\end{theorem}

A harmonic function $u$ defined on $\mathbb R^n$ for $n\geq 2$ is said to be universal 
for uniform convergence
if the set of its translates $\{u(\cdot+a_j):a_j\in\mathbb R^n\}$  is dense in the space of all harmonic functions on $\R^n$  with the topology of local uniform convergence.

\begin{definition}
A measurable function $u$ on $\R^n$ is said to be universal 
for convergence in measure if for every measurable function $v$ there is a sequence $a_j\in \R^n$ such that the sequence of translates $u(\cdot+a_j)$ converges to $v$ in measure on compact subsets.  
\end{definition}

 In our next result we show that universality for uniform convergence is stronger than universality for convergence in measure.

\begin{theorem}
\label{harmonicuniversal}
Every harmonic function, which is universal for uniform convergence is also universal for convergence in measure.
\end{theorem}

\begin{proof} 
On $\R^n,$ let $u$ be a universal harmonic function universal for uniform convergence
and $v$ a measurable  function. Fix a compact set $K\subset \R^n$ and $\epsilon>0.$ By Theorem \ref{GS}, there is a sequence $\{u_j\}_{j=1}^\infty$,  $u_j\in H(\R^n),$ such that
$u_j\to v$ in measure. Since $u$ is universal in the sense of Birkhoff, there are $a_j\in \R^n,$, $j=1,2,\ldots$, such that 
$$
	|u(x+a_j)-u_j(x)|<\frac{1}{j},  \quad \mbox{for all} \quad \|x\|<j. 
$$
Choose $j_\epsilon> \|x\|,$ for all $x\in K.$ Then, for all $j>j_\epsilon,$ 
$$
	\mu\{x\in K: |u(x+a_j)-v(x)|>\epsilon\} \le 
	\mu\{x\in K:|u_j(x)-v(x)|>\epsilon-1/j\}.
$$
Since the $u_j\to v$ in measure, the right hand side tends to $0$ as $j\to\infty.$ Therefore $u(\cdot+a_j)$ converges to $v$ in measure. 

\end{proof}

Recall that a subset of a topological space is said  to be {\em meagre} if it is a countable union of nowhere dense sets. Meagre sets are also called sets of first Baire category. Sets which are not meagre are called sets of second Baire category. A subset $A$ of a topological space $X$ is said to be {\em residual} in $X,$ if $X$ is of second category while $X\setminus A$ is of first category. Thus, not only is the complement $X\setminus A$ meagre, but $A$ is {\em not} meagre (since the union of two meagre sets is obviously meagre). In a topological sense we may consider that  a residual subset of a topological space contains ``most" points of the space.
The space $H(\R^n)$ of harmonic functions is of second Baire category (by the Baire Category Theorem) and the family of functions in $H(\R^n)$ which are universal for uniform convergence
is known  (see e.g \cite{AG96}) to be residual in $H(\R^n).$  Since a set containing a residual set is clearly residual, we immediately have the following corollary.

\begin{corollary} 
\label{Corollary:universal}
The set of harmonic functions on $\R^n$ which are universal for convergence 
in measure is residual in the space $H(\R^n)$ of all harmonic functions on $\R^n.$ 
\end{corollary}

Theorem \ref{DirichletTheorem}, Theorem \ref{harmonicuniversal} and Corollary \ref{Corollary:universal} are valid if we consider functions of a single complex variable instead of harmonic functions. Thus, all the results concerning harmonic functions on Riemannian manifolds hold for holomorphic functions on an open Riemann surface endowed with a Riemannian metric, see \cite[Theorem 4.4]{FalcoGauthier}. Moreover, we have the analogous results for convergence in measure for entire universal functions.
All proofs are the same, since all the tools from harmonic approximation, which which we have employed have analogues for holomorphic approximation.

Similar results hold for functions of several complex variables, but the situation in several complex variables is more delicate, since the underlying approximation theory is less developed. This will be considered in a future paper.


\section{Universality of Dirichlet series for convergence in measure}
\label{UniversalDirichletseries}

In this section, 
we study universality for convergence in measure 
of Dirichlet series defined on $\mathbb C$. We consider functions $\mathcal L(s)$ that  have a representation as a Dirichlet series 
\[
	\mathcal L(s)=\sum_{n=1}^{\infty}\frac{a(n)}{n^{s}},
\]
in some half-plane $\sigma>\sigma_0,$ where $\sigma = \Re s.$ 
The class of functions  $\tilde{\mathcal S},$ 
sometimes called the Steuding class (see \cite{S}), 
consists of all  Dirichlet series that satisfy:
\begin{itemize}
\item[(i)] \textit{Holomorphic continuation:} There exists a real number $\sigma_{\mathcal L}$ such that $\mathcal L(s)$ has a holomorphic continuation to the half-plane $\sigma > \sigma_{\mathcal L}$ with $\sigma_{\mathcal L} < 1$ except for at most a pole at $s = 1$,
\item[(ii)] \textit{Finite order:} There exists a constant $\mu_{\mathcal L}\geq 0$ such that, for every fixed $s$ with $\Re s>\sigma_{\mathcal L}$ and every $\varepsilon>0$,
\[
	\mathcal L(s+it)\ll \vert t\vert^{\mu_{\mathcal L}+\epsilon}\ \ \ \text{ as }\ \ \ |t|\to\infty
\] where  the implicit constant may depend on $\epsilon$,
\item[(iii)] \textit{Polynomial Euler product:} There exists a positive integer $m$ and for every prime $p$, there are complex numbers $\alpha_{j} (p$), $1 \leq  j \leq m$, such that
\[
	\mathcal L(s)=\prod_{p}\prod_{j=1}^{m}\left(1-\frac{\alpha_{j}(p)}{p^{s}}\right)^{-1},
\]
\item[(iv)] \textit{Prime mean-square:} There exists a positive constant $k$ such that
\[
	\lim_{x\to\infty} \frac{1}{\pi(x)}\sum_{p\leq x}\vert a(p)\vert^{2}=k.
\]
\end{itemize}

Steuding's definition of the class $\widetilde{\mathcal S}$ includes an additional condition, but (as Steuding remarks) it is a consequence of the other conditions. 

Note that the class $\widetilde{\mathcal S}$ is not empty. In particular the Riemann zeta-function $\zeta(s)$, the Dirichlet $L$-functions and the Dedekind zeta-functions are in $\tilde{\mathcal S}$.
 
For a function $\mathcal L$ in the class $\widetilde{\mathcal S}$, 
we consider convergence on vertical lines of the mean-square
$$
	\frac{1}{2T}\int_T^T|\mathcal L(\sigma+it)|^2dt
$$
and we denote by $\sigma_{m}$ the  abscissa of the mean-square half-plane of $\mathcal L$ (see \cite[Section 2.1]{S}).

In 1995, Voronin proved his spectacular theorem on the universality for uniform convergence of the
Riemann zeta-function and later Steuding  extended Voronin's theorem to the class $\widetilde{\mathcal S}$ of Dirichlet series  as follows (see also \cite{M}). 

\begin{theorem}\label{Steuding}\cite[Theorem 5.14]{S}
Let $\mathcal L\in \tilde S, \, K$ be a compact subset of the strip $\frak S=\{ \sigma_m<\sigma <1\}$ with connected complement, and let $g$ be a zero-free continuous function on $K$ which is holomorphic in the interior of $K.$ Then, for every $\epsilon > 0,$ 
$$
	\liminf_{T\to\infty}\frac{1}{T}
\mbox{meas}\left\{t\in[0,T]:\max_{s\in K}|\mathcal L(s+it)-g(s)|<\epsilon\right\} > 0.
$$ 
\end{theorem}

For the Riemann zeta-function it follows from  \cite[Theorems 7.2 and 7.3]{T} that $\sigma_m=1/2.$ 
Thus, for the Riemann zeta-function on the right-half of the fundamental strip 
we have universality for uniform convergence, which is precisely Voronin's theorem.

\begin{remark}\label{remark}
Note that the strip of universality $\frak S = (\sigma_{m}<\sigma<1)$ of $\mathcal L$ is not empty since 
$$
\sigma_{m} \le 
\max\left\{\frac{1}{2},1-\frac{1-\sigma_{\mathcal L}}{1+2\mu_{\mathcal L}}\right\} < 1,
$$ 
(see \cite[Theorem 2.4]{S}). 
\end{remark}
  
Our first result in this section establishes universality for convergence in measure of functions in the  class $\widetilde{\mathcal S}$.

\begin{theorem}\label{S-measure}
Let $\mathcal L$ be a function in $\tilde{\mathcal S}.$ 
Consider the strip  $\frak S = (\sigma_{m}<\sigma<1),$ a bounded Borel measurable subset $A$ of $\frak S$   and a measurable function $f$ on $A.$ Then, there exists a sequence $\{t_n\}$ of real numbers, such that $\mathcal L(\cdot+it_n)$ converges to $f$ in measure. 
\end{theorem}

\begin{proof}
Let $m(E)$ denote the Lebesgue 2-measure of a planar set $E.$ 
We may assume that $m(A)>0,$ since otherwise the theorem is trivial. Multiplying by a constant we can assume that $m(A)>1.$ 
By Luzin's theorem, 
there exists a sequence  $\{A_n\},$ of compact subsets of $A,$ such that, for each $n,$ the restriction $f_n$ of $f$ to $A_n$ is continuous and $m(A_n)>m(A)-1/n.$ By \cite[Theorem 1.2]{GS17}, there is a compact subset $F_n\subset A_n,$ such that $R\setminus F_n$ is connected and $m(F_n)>m(A_n)-1/n.$ Notice that $\mathbb C\setminus F_n$ is also connected. 

Denote by  $\mathcal P_n$ a dyadic partition of $\C$ in squares of side length 
$2^{-k}$, where $k$ will be determined latter and, by abuse of notation, we use the same notation for the associated covering of $\C$ by closed squares.  Let $\mathcal Q_n$ be the collection of those (finitely many  by the compactness of $F_{n}$) squares $Q$ of this partition such that $m(F_n\cap Q)>0.$ For each fixed $n,$  by the compactness of the set $F_{n}$ and the continuity of the function $f$ we can choose  $k$ such that the partition $\mathcal P_n$  is so fine that
$$
	\max\{|f(z)-f(w)|: z,w\in F_n\cap Q\}<1/n, \quad \forall \quad Q\in \mathcal Q_n.
$$ 
Since $\mathcal Q_n$ consists of only finitely many squares $Q,$ for each such $Q$ we may choose a compact square $\widetilde Q\subset Q^0,$ such that defining $K_n=\bigcup_Q F_n\cap \widetilde Q$,  we have $m(K_n)>m(F_n)-1/n.$ 

We define a function $g_n$ holomorphic on (a neighbourhood of) $K_n$ by  choosing an arbitrary point $a_Q$ in each $F_n\cap\widetilde Q,$  choosing a value $b_Q\not=0$ such that $|f(a_Q)-b_Q|<1/n$ and finally setting $g_n=b_Q$ on each $F_n\cap \widetilde Q.$  That is $g_n=\sum_Q b_Q \chi_{F\cap \widetilde Q}$.

For each $z\in K_n,$ if $Q$ is the unique $Q\in \mathcal Q_n$ such that $z\in F_n\cap\widetilde Q,$ we have 
\begin{equation}\label{K_n}
	|f(z)-g_n(z)|=|f(z)-b_Q|  \le |f(z)-f(a_Q)|+|f(a_Q)-b_Q|<2/n. 
\end{equation}
Since $\mathbb C\setminus K_n$ is connected, it follows from Theorem \ref{Steuding} that there is a real number $t_n,$ such that 
\begin{equation}\label{Voronin}
	\max\{|g_n(z)-\mathcal L(z+it_n)|: z\in K_n\} < 1/n, \quad n=1,2,\ldots. 
\end{equation}
From equations \eqref{K_n} and \eqref{Voronin},
$$
	|f(z)-\mathcal L(z+it_n)|<3/n, \quad \mbox{for all} \quad z\in K_n.
$$
Since $K_n\subset F_n\subset A_n \subset A,$ we have 
$$
m(A\setminus K_n)=m(A\setminus A_n)+m(A_n\setminus F_n)+m(F_n\setminus  K_n)<3/n.
$$
Thus, 
$$
	m\{z\in A:|f(z)-\mathcal L(z+it_n)|>3/n\} < 3/n. 
$$
It follows that $\mathcal L(\cdot+it_n)\rightarrow f$ in measure on $A.$ 
\end{proof}

As a particular case of Theorem \ref{S-measure}, we have  universality in measure for the Riemann zeta-function.

\begin{corollary}
On the strip $R = (1/2<\Re z<1),$
let $m$ be the Lebesgue measure, $A$ be  a bounded Borel measurable subset of $R$ and $f$ be a measurable function on $A.$ Then, there exists a sequence $\{t_n\}$ of real numbers, such that $\zeta(\cdot+it_n)$ converges to $f$ in measure. 
\end{corollary}

We shall now present an analogue of Theorem \ref{Steuding}, asserting  that, in the strip $\frak S=(\sigma_{m}<\sigma<1),$ functions in the class $\widetilde{\mathcal S}$ are universal  for convergence in measure. This is indeed stronger than  Theorem \ref{S-measure} and an alternative proof of Theorem \ref{S-measure} can be obtained by considering values $t_n$ obtained by applying this result for $\epsilon=n^{-1}$ for $n=1,2,\ldots$

\begin{theorem}\label{Voronin measure}
Let $\mathcal L\in \tilde S, \, A$ be a bounded measurable  subset of the strip $\frak S=\{ \sigma_m<\sigma <1\}$ and let $\phi$ be a measurable function on $\mathcal A.$ Then, for every $\epsilon > 0,$ 
$$
	\liminf_{T\to\infty}\frac{1}{T}
	\mbox{meas}_1\big\{t\in[0,T]:\mbox{meas}_2\{s\in A:|\mathcal L(s+it)-\phi(s)>\epsilon\}
	<\epsilon\big\} > 0,
$$
where $\mbox{meas}_1$ and $\mbox{meas}_2$ are respectively Lebesgue 1-measure and 2-measure. 
\end{theorem}

\begin{proof}
In the proof, we shall replace $\mbox{meas}_2$ by $m_2$ for brevity of notation. Let $\phi$ be a measurable function on a bounded  subset $\mathcal A$ of the strip $\frak S$ and choose  $\epsilon > 0.$ By Luzin's theorem, there is a compact subset 
$C\subset \mathcal A,$ such that the restriction  $\phi_C$ of $\phi$ on $C$ is continuous and $m_2(A\setminus C)<\epsilon/3.$ Moreover, we may assume that the complement of $C$ is connected. Since  $\phi_C$ is  continuous on $C$ and $C$ is compact, by uniform continuity there is a $\delta>0,$ such that, 
$$
	|s-s^\prime|<\delta, \quad s, s^\prime \in C \quad \implies \quad 
		|\phi_C(s)-\phi_C(s^\prime)|<\epsilon/3. 
$$
Consider a dyadic partition of the plane $\C$ so  that the diameters of each resulting square $Q$ is less than $\delta$.   Thus, for each square $Q$ of the partition, 
$$
	s, s^\prime \in  C\cap Q  \implies |\phi_Q(s)-\phi_Q(s^\prime)|<\epsilon/3.
$$

Let $Q_1, Q_2, \ldots, Q_n$ be the squares of the partition whose interior meet $C.$  For each $j=1, 2, \ldots, n,$ choose a point $s_j\in C\cap Q^0$ and let $\tilde Q_j$ be a closed square such that $s_j\in \tilde Q_j\subset Q_j^0$ and $m_2(\mathcal C\cap\tilde Q_j) > m_2(C\cap Q_j)-\epsilon/(3n).$ For $j=1, 2, \ldots, n,$ set $K_j = C\cap\tilde Q_j$ and  $K=K_1\cup K_2\cup\ldots\cup K_n.$ Then, $K$ is a compact subset of the strip $\frak S$ and the complement of $K$ is connected. 
For each $j=1, 2, \ldots, n,$ choose a value $v_j\not=0,$ such that $|v_j-\phi(s_j)|<\epsilon/3.$ We define a function $g$ continuous on  $K$ and  holomorphic on $K^\circ$  by setting $g(s) = v_j,$ for $s\in K_j.$ We have 
$$
	|g(s)-\phi(s)| < 2\epsilon/3, \quad \mbox{for all} \quad s\in K
$$
and 
$$
	m_2(\mathcal A\setminus K) < 2\epsilon/3.
$$
Now apply Theorem \ref{Steuding}, for this $K,$ this function $g$ and $\epsilon$ replaced by $\epsilon/3.$

\end{proof}

We conclude with some considerations regarding the Riemann Hypothesis.

\begin{theorem}\label{Bagchi}\cite{B}
The Riemann Hypothesis holds if and only if the Riemann zeta-function approximates itself 
uniformly in the sense of Theorem \ref{Steuding} on the right half $(1/2<\sigma<1)$ of the fundamental strip. 
\end{theorem}

\begin{theorem}
The hypothesis that the function to be approximated is zero-free cannot be omitted in Theorem \ref{Steuding}.
\end{theorem}

\begin{proof}
Suppose we could omit the zero-free hypothesis. There would be two
incompatible consequences. \\
1) By Theorem \ref{Bagchi} the Riemann Hypothesis would hold. \\
2) We could approximate the function $s-3/4,$ which invoking Rouché's theorem would imply that the Riemann zeta-function has a zero in the fundamental strip. So the Riemann Hypothesis would fail. 
\end{proof}

{\bf Remark.} The Riemann zeta-function does approximate itself in the sense of Theorem \ref{Voronin measure} (a weak measure-theoretic analogue of Theorem \ref{Steuding}) on the right half $(1/2<\sigma<1)$ of the fundamental strip. 

{\bf Remark.} In Theorem \ref{Voronin measure} there is no assumption that the function to be approximated is bounded away from zero. 

In view of the last theorems and remarks, establishing a stronger version of Theorem \ref{Voronin measure} might yield new information regarding the Riemann Hypothesis. For example, consider the following conjecture which involves the Selberg class. We recall that the Selberg class is the set of Dirichlet series satisfying:
\begin{itemize}
\item[(i)] \textit{Ramanujan hypothesis:}  $a(n) \ll n^\epsilon$ for every $\epsilon > 0$, where the implicit constant may depend on $\epsilon$,
\item[(ii)] \textit{Holomorphic continuation:} There exists a non-negative integer $k$ such that $(s- 1)^k \mathcal L(s)$ is an entire function of finite order,
\item[(iii)] \textit{Functional equation:}  $\mathcal L(s)$ satisfies a functional equation of type
\[
	\Lambda_{\mathcal L}(s)=\omega\overline{\Lambda_{\mathcal L}(1-\overline{s}})
\]
where
\[
	\Lambda_{\mathcal L}:=\mathcal L(s)Q^s\prod_{j=1}^f\Gamma(\lambda_js+\mu_j)
\]
with positive real numbers $Q$, $\lambda_j$, and complex numbers $\mu_j,\omega$ with $\Re \mu_j\geq0$ and $\vert \omega\vert=1$.
\item[(iv)] \textit{Euler product:} $\mathcal L(s)$ has a product representation
\[
	\mathcal L(s)=\prod_p \mathcal L_p(s)
\]
where 
\[
\mathcal L_p(s)=exp\Big(\sum_{k=1}^\infty \frac{b(p^k)}{p^{ks}}\Big)
\]
with suitable coefficients $b(p^k)$ satisfying $b(p^k)\ll p^{k\theta}$ for some $\theta < 1/2$.
\end{itemize}

\begin{conjecture}\label{Steuding fewer zeros} Let $S$ be the Selberg  and
let $\mathcal L\in \widetilde{\mathcal S}\cap S, \, K$ be a compact subset of the strip $\frak S=\{ \sigma_m<\sigma <1\}$ with connected complement, and let $g$ be a continuous function on $K,$ holomorphic in $K^0$ and having no isolated zeros in $K^0.$
Then, for every $\epsilon > 0,$ 
$$
	\liminf_{T\to\infty}\frac{1}{T}
\mbox{meas}\left\{t\in[0,T]:\max_{s\in K}|\mathcal L(s+it)-g(s)|<\epsilon\right\} > 0.
$$ 
\end{conjecture}

Our hope that this conjecture would provide information on the Riemann Hypothesis may or may not be justified, however this conjecture is certainly interesting for the following reason. Johan Andersson \cite{A13} has shown the remarkable fact that, for the case where $\mathcal L$ is the Riemann zeta-function, this conjecture is equivalent to the following conjecture regarding a natural strengthening of the most important theorem on polynomial approximation, Mergelyan's theorem. 

\begin{conjecture} \label{Mergelyan fewer zeros}
Let   $K$ be a compact subset of $\C$ with connected complement, and let $g$ be a continuous function on $K,$ holomorphic in $K^0$ and having no isolated zeros in $K^0.$ Then, there is a sequence of polynomials $p_j,$ zero-free on $K,$ such that $p_j\to f$ uniformly on $K.$ 

\end{conjecture}
We show now that Andersson's result extend to functions in the intersection $\widetilde{\mathcal S}\cap S$ of the Steuding class with the Selberg class. 

\begin{theorem} Conjecures \ref{Steuding fewer zeros} and \ref{Mergelyan fewer zeros} are equivalent.
\end{theorem}

\begin{proof} To show that Conjecture \ref{Mergelyan fewer zeros} implies Conjecture \ref{Steuding fewer zeros} note that Anderson's argument used for the case of the Riemann zeta-function also works for the general case $\mathcal L \in \tilde S$ if we replace in the original proof of Anderson's the use  of the Voronin Universality Theorem by Theorem \ref{Steuding}.

We show that  Conjecture \ref{Steuding fewer zeros} implies Conjecture \ref{Mergelyan fewer zeros}. It follows from the zero-density results  of Kachotowdki and Perelli (see \cite[Ch. 8, p. 160]{S}) that There is a $\sigma_*,$ with $\sigma_m<\sigma_*<1,$ such that 
\begin{equation}\label{o}
	N_\mathcal L(\sigma_*,T) = o(T),
\end{equation}
where 
$$
	N_\mathcal L(\sigma_*,T) = \#\{s: \mathcal L(s)=0, \, \Re s\ge\sigma_*, \, 0\le\Im s\le T\}.
$$

For fixed $m\in \N,$ set $I_j=\{t:(j-1)m<t<jm\},$ and denote by $\nu(n)$ the number of intervals $I_j,$ among $I_1, I_2, \ldots, I_n$ for which the function $\mathcal L$ is zero free on $\{s: \sigma_*\le \sigma, \,  t\in I_j\}.$  

We claim that for each fixed $0<\lambda<1,$ there exists an $n_{m,\lambda},$ such that 
$\nu(n)/n>\lambda,$ for all $n>n_{m,\lambda}.$ Suppose this is not the case. Then, there are arbitrarily large values of $n,$ such that on the interval $[0,mn]$ there are at least $\lambda$ many disjoint intervals of length $m,$ on each of which $\mathcal L$ has a zero. Thus, for arbitrarily large values of $n,$
$$
	\frac{N_\mathcal L(\sigma_*,mn)}{mn} \ge \frac{n-\nu(n)}{mn}\ge \frac{1-\lambda}{m} > 0.
$$ 
This contradicts (\ref{o}) and establishes the claim.

Imitating the proof in \cite{A13}, we can choose  a linear transformation $\ell,$ such that $K_*=\ell(K)$ lies in the rectangle $\frak S^+_*=\{s: \sigma_*<\sigma<1, \, 0< t < m \}.$ Fix $\epsilon>0.$ If Conjecture  \ref{Steuding fewer zeros} holds, then  there is a $\delta>0,$ such that for all sufficiently large $n,$ 
\begin{equation}\label{delta}
	\frac{1}{mn}
	meas\{t\in [0,mn]:  \max_{s\in K_*}|\mathcal L(s+it)-g(\ell^{-1}(s))|>\epsilon\} 
	> \delta. 
\end{equation}
According to the claim, we can choose $\nu(n),$ such that $\nu(n)/n>1-\delta,$ for all $n>n_{m,1-\delta}.$ Thus, 
$$
	\frac{1}{mn}
	meas\{t\in[0,mn]; s\in(\sigma_*,1)\times[0,m], \, \mathcal L(s+it)\not=0\} = \frac{\nu(n)\cdot m} 	{mn}>1-\delta. 
$$
In particular, 
\begin{equation}\label{1-delta}
	\frac{1}{mn}
	meas\{t\in[0,mn]; s\in K_*, \, \mathcal L(s+it)\not=0\} >1-\delta. 
\end{equation}
It follows from (\ref{delta}) and (\ref{1-delta}) that there is a $t\in [0,mn]$ such that $\mathcal L(s+it)\not=0$ for all $s\in K_*$ and $|\mathcal L(s+it)-g(\ell^{-1}(s))|<\epsilon,$ for all $s\in 
K_*.$  Thus, $|\mathcal L(\ell(z)+it)-g(z)|<\epsilon,$ for $z\in K.$ By Mergelyan's theorem, there is a sequence of polynomials $p_k(z)$ which converges uniformly to  $\mathcal L(\ell(z)+it)$ on $K.$ For sufficiently large $k,$ $p_k$ is zero-free on $K,$ and $|p_k(z)-\mathcal  L(\ell(z)+it)|<\epsilon,$ for all $z\in K.$ By the triangle inequality, we have a polynomial $p,$ which is zero-free on $K$ and such that $|p(z)-g(z)|<2\epsilon,$ for all $z\in K.$ Thus Conjecture \ref{Mergelyan fewer zeros} holds. 
\end{proof}

\smallskip
For a Riemann surface $M,$ we denote by $Hol(M)$ the family of functions holomorphic on $M.$
The following is a complex analogue of Theorem \ref{GS}.

\begin{theorem}[{\cite[Corollary 1.4]{GS17}}]\label{GS-complex}
For every measurable subset $E$ of a non-compact Riemannian surface $M,$ and for every
measurable function $v : E\rightarrow \mathbb R,$ there exists a sequence $u_j\in Hol(M),$ such that $u_j \rightarrow v$ in measure.
\end{theorem}

The following corollary is related to the previous considerations. 

\begin{corollary}
Let $K$ be a compact subset of $\C$ having connected complement and $g$ a function continuous on $K,$ holomorphic on $K^0$ and having no isolated zeros in $K^0.$ Then, there is a sequence of polynomials $p_j,$ zero-free on $K,$ such that $p_j\to g$ in measure on $K.$ 
\end{corollary}

\begin{proof}
If $m_2(f^{-1}(0))=m_2(K),$ we may set $p_j=1/j.$ 

Suppose $m_2(f^{-1}(0)) < m_2(K)$ and fix $\epsilon>0.$ For each natural number $j$ define $A_j=\{z\in K: |g(z)|\le 1/j\}$ and $B_j=\{z\in K: |g(z)|\ge 2/j\}$. Fix $j$ sufficiently large so that $m_2(A_j\cup B_j)>m_2(K)-\epsilon.$ Let $\{U_k\}$ be an enumeration of the the complementary components of $A_j$ and $B_j$ and let $\{\delta_k\}$  be a sequence of positive numbers. For each $k,$ we let $R_k$ be an open rectangle connecting $U_k$ to $\C\setminus K$ with measure smaller than $\delta_k$ and we set $K_\epsilon=A_j\cup B_j\setminus \cup_k R_k.$ Then, $K_\epsilon$ is a compact subset of $K$ with connected complement and moreover $K\cap A_j$ and $K\cap B_j$ have connected complements. We choose the $\delta_k's$ so small that $m_2(K\setminus K_\epsilon)<2\epsilon.$ We define a function $g_\epsilon \in A(K_\epsilon)$ by setting $g_\epsilon = 1/j$ on $K_\epsilon\cap A_j$ and $g_\epsilon =g$ on $K_\epsilon$ on $K_\epsilon\cap B_j.$ By Mergelyan's theorem, there is a polynomial $p_\epsilon,$ such that $|p_\epsilon-g_\epsilon|<\epsilon/2$ on $K_\epsilon.$ 
We note that $|p_\epsilon|>\epsilon/2$ on $K_\epsilon.$ 
Now, choosing a sequence $\epsilon_n\to 0,$ we obtain a sequence $p_n$ of polynomials, zero-free on $K,$  which converges in measure to $g$ on $K.$ 
\end{proof}


\end{document}